\documentclass[12pt]{article}
\usepackage{amsmath}
\usepackage{amsfonts}
\usepackage{pb-diagram}
\usepackage{amsthm} 
\usepackage{amscd}
\newcounter{numb}

\theoremstyle{plain} %% This is the default
\newtheorem{theorem}{Theorem}[section]
\newtheorem{corollary}[theorem]{Corollary}

\newtheorem{proposition}[theorem]{Proposition}
\newtheorem{definition}[theorem]{Definition}
\newtheorem{remark}[theorem]{Remark}
\newtheorem{example}[theorem]{Example}
\numberwithin{equation}{section}

\newcommand{\bz}{\mathbf {Z}}
\newcommand{\x}{\mathcal{X}}

\newcommand{\w}{\wedge}

\newcommand{\p} {\partial}
\newcommand{\g} {\mathfrak {g}}
\newcommand{\kl}{\mathfrak {K}}

\newcommand{\G}{\mathfrak {G}}
\newcommand{\F}{\mathfrak {F}}
\begin{document}
\title{ LIE $n$-RACKS }
\author{Guy Roger Biyogmam}
%\footnote{The author gives special thanks to  \textbf{Dr. Jerry Lodder}.}}
\date{}
\maketitle{Department of Mathematics,\\Southwestern Oklahoma State University,\\100 Campus Drive,\\Weatherford, OK 73096, USA,\\\texttt{Email:}\textit{guy.biyogmam@swosu.edu}.}

\begin{abstract} 
 %Kinyon showed that the tangent space of a Lie Rack at the neutral element has a Leibniz algebra structure. This provided a promising lead towards solving the Coquecigrue  problem for Leibniz algebras.
 In this paper, we introduce the category of  Lie $n$-racks and generalize several results known on racks. In particular, we show that the tangent space of a Lie $n$-Rack at the neutral element has a Leibniz $n$-algebra structure. We also define a cohomology theory of $n$-racks.\\
\end{abstract}

\section{Introduction and Generalities}

 Manifold with $n$-ary operations on their function algebra have been an important area of attraction for many years. It all started with Nambu \cite{N} introducing in  1973, an $n$-ary generalization of Hamiltonian Dynamics by the $n$-ary Poisson bracket, also known as the Nambu bracket. Since then, Nambu-Poisson structures have been seriously studied. The work of   Takhtajin \cite{T} in 1994 on the foundation of the theory of Nambu-Poisson manifolds, and recent applications of Nambu brackets to brane theory \cite{B}  are just a few to mention. In particular, Nambu algebras are infinite dimensional Filippov algebras. Also known in literature as $n$-Lie algebras (not to mistake with Lie $n$-algebras),  Filippov algebras  were first introduced in 1895 by Filippov \cite{F} and was generalized to the concept of Leibniz $n$-algebras by Casas, Loday and Pirashvili \cite{CLP}. Both concepts are important in Nambu Mechanics \cite{N}. Applications of Filippov algebras are found in String theory \cite{B} and in  Yang-Mills theory \cite{FU}.

One of the most important problems in Leibniz algebra theory is the coquecigrue problem (a generalization of Lie's  third theorem to Leibniz algebras) which consists of finding  a  generalization of groups whose tangent algebra structure corresponds to a Leibniz algebra. Loday dubbed these objects ``coquecigrues'' \cite{LD} as no properties were foreseen on them. While attempting  to solve this problem, Kinyon \cite{K} showed that  the tangent space at the neutral element of a Lie rack has a Leibniz algebra structure. In his doctoral thesis, Covez \cite{C} provided a local answer to the coquecigrue  problem  by showing that every Leibniz algebra becomes integrated into a local augmented Lie rack. Thus Lie racks seem to be the best candidates for ``coquecigrues''. 

   Meanwhile, Grabowski and Marmo  \cite{GM} provided in the same order of idea an important connection between Filippov algebras and Nambu- Lie Groups. All these ideas suggest a new mathematical structure by extending the binary operation of Lie racks to an $n$-ary operation. This yields the introduction of Lie $n$-racks and generalizes Lie racks from the case $n=2.$ It turns out that one can extend Kinyon's result to Leibniz algebras via Lie $n$-racks. If it is reasonable to talk about the coquecigrue problem for Leibniz $n$-algebras, this work certainly opens a lead towards a solution.

In \cite{FRS}, Fenn, Rourke and sanderson introduced a cohomology theory for racks which was modified in \cite{CJKLS} by Carter, Jelsovsky, Kamada, Landford and Saito to obtain quandle cohomology, and several results have been recently established. We use these cohomology theories in section 4 to define  cohomology theories on $n$-racks and $n$-quandles. 

let us recall a few definitions. Given a field $\kl$ of characteristic different to 2, a Leibniz $n$-algebra \cite{CLP} is defined as a $\kl$ -vector space $\g$ equipped with an $n$-linear operation\\ $[-,\ldots,-]: \g^{\otimes n}\longrightarrow \g$ satisfying the identity 

 $$[x_1,\ldots, x_{n-1}, [y_1,y_2,\ldots, y_n]]=\sum_{i=1}^n[y_1,\ldots,y_{i-1},[x_1,\ldots,x_{n-1},y_i], y_{i+1},\ldots, y_n]~~~(1.1)$$
Notice that in the case where the $n-ary$ operation $[-,\ldots,-]$ is antisymetric in each pair of variables, i.e., $$[x_1,x_2,\ldots,x_i,\ldots,x_j,\ldots,x_n]=-[x_1,x_2,\ldots,x_j,\ldots,x_i,\ldots,x_n]$$ or equivalently $[x_1,x_2,\ldots,x,\ldots,x,\ldots,x_n]=0~\mbox{for all $x \in$ G},$ the Leibniz $n$-algebra becomes a Filippov algebra (more precisely a $n$-Filippov algebra). Also, a Leibniz $2$-algebra is exactly a Leibniz algebra \cite[p.326]{L} and become a Lie algebra if the binary operation $[~,~]$ is skew symmetric.

If $\g$ is a vector space endowed with an $n$-linear operation $\sigma:\g\times\g\times\ldots\times\g\longrightarrow\g$, then a map $D:\g\longrightarrow\g$ is called a derivation with respect to $\sigma$ if 
$$D(\sigma(x_1,\ldots,x_n))=\sum_{i=1}^n\sigma(x_1,\ldots,D(x_i),\ldots,x_n).$$

 A Nambu-Lie group $(G,P)$ is a Lie group $ G$  with a rank $n$  $G$-multiplicative Nambu tensor Nambu tensor $P$ i.e.,$$P(g_1,g_2)=L_{g_1^*}P(g_2)+R_{g_2^*}P(g_1),$$ where $L_{g_1^*}$ and $R_{g_2^*}$ denote respectively left and right  translations in $G.$ More on Nambu manifolds can be found in \cite{V}

The following proposition provides an important connection between Nambu-Lie groups and Filippov algebras. It is mainly stated in this paper because it shows an analogy between the connection of filippov algebras with Nambu-Lie groups, and the connection of Leibniz $n$-algebras with Lie $n$-racks provided in section 4.
\begin{proposition} ( \cite {GM})
Let $(G,P)$ be a Nambu-Lie group and let $\delta_p:\g\longrightarrow \w^n(\g)$ denote the intrinsic derivative $\delta_P(X)=L_X(P)(e)$ where $L_X$ denotes the Lie derivative and $e$ the identity element of  $G.$ Then the map $\delta^*_P:\w^n(\g^*)\longrightarrow\g^*$ defines a filippov bracket on $\g^*.$ 
\end{proposition}

A lie Rack $(R,\circ,1)$ is a smooth manifold R with a binary operation $\circ$ and a specific element $1\in R$ such that the following conditions are satisfied:
\begin{itemize}
\item $x\circ(y\circ z)=(x\circ y)\circ (x\circ z)$
\item for each $x,y\in R$, there exits a unique $a\in R$ such that $x\circ a= y$
\item $1\circ x=x$  and $x\circ 1=1$ for all $x\in R$ 
\item the operation $\circ: R\times R\longrightarrow  R$ is a smooth mapping. 
\end{itemize}

\section {$n$-racks}

\begin{definition}: A left  $n$-rack\footnote{2-racks coincide with Racks. They   were introduced  in 1959 by G. Wraith and J. Conway\cite{CW}, and have been used in various topics of mathematics. In particular, they are mainly used in topology to provide invariants for knots \cite{DJ}.} (right $n$-racks are defined similarly) $(R,[-,\ldots,-]_{R})$ is a set $R$ endowed with an $n$-ary operation $[-,\ldots,-]_{R}: R\times R\times\ldots\times R\longrightarrow R$  such that 

\begin{enumerate}
\item $[x_1,\ldots,x_{n-1},[y_1,\ldots,y_{n-1}]_{R}]_{R}=[[x_1,\ldots,x_{n-1},y_1]_{R},\ldots,[x_1,\ldots,x_{n-1},y_n]_{R}]_{R}~~~(2.1)$\\(This is the left distributive property of $n$-racks) 
\item For   $a_1,\ldots,a_{n-1},b\in R,$ there exists  a unique $x\in R$ such that $[a_1,\ldots,a_{n-1},x]_{R}=b.$
\end{enumerate}
If in addition there is a distinguish element  $1\in R$, such that $$\textit{3.}~[1,\ldots,1,y]_{R}=y~~\mbox{and}~~[x_1\ldots,x_{n-1},1]_{R}=1~~\mbox{ for all}~~ x_1,\ldots,x_{n-1}\in R,~~~~~~~~~~$$ then $(R,[\ldots]_R,1)$ is said to be a pointed $n$-rack.

A $n$-rack is a weak $n$-quandle if it further satisfies   $$[x,x,\ldots,x,x]_R=x~\mbox{for all}~ x\in R.$$ 

A $n$-rack is a $n$-quandle if it further satisfies   $$[x_1,x_2,\ldots,x_{n-1},y]_R=y~\mbox{if}~ x_i=y~~\mbox{for some}~ i\in \{1,2,\ldots,n-1\}.$$ 

A $n$-quandle (resp weak $n$-quandle) is a $n$-kei (resp weak $n$-kei) if it further satisfies $$[x_1,\ldots,x_{n-1},[x_1,\ldots,x_{n-1},y]]=y~~\mbox{for all}~x_1,\ldots,x_{n-1},y\in R.$$  
\end{definition}
For $n=2,$ one recovers racks, quandles\cite{DJ} and keis \cite{TA}. Note also that $n$-quandles are also weak $n$-quandles, but the converse is not true for $n>2;$ See example 2.3.

\begin{definition}
Let $R,~R'$ be $n$-racks. A function $\alpha: R\longrightarrow R'$ is said to be a homomorphism of $n$-racks if $$\alpha([x_1,\ldots,x_n]_R)=[\alpha(x_1),\alpha(x_2),\ldots,\alpha(x_n)]_{R'}~~\mbox{for all}~x_1,x_2,\ldots,x_n\in R.$$

\end{definition}
We may thus form the category $_npRACK $ of pointed  $n$-racks and pointed $n$-rack homomorphisms.

\begin{example}
A $\bz_4$-module $M$  endowed with the operation $[-,\ldots,-]_M$ defined by   $$[x_1,\ldots,x_n]_M=2x_1+2x_2+\ldots +2 x_{n-1}+x_n $$ is a $n$-rack that  is a weak $n$-kei  if $n$ is odd.
Indeed $$[[x_1,\ldots,x_{n-1},y_1]_M,\ldots[x_1\ldots,x_{n-1},y_n]_M,]_M=~~~~~~~~~~~~~~~~~~~~~~~~~~~~~~~~~~~~~~~~~~$$
$$~~~~~~~~=(\sum_{i=1}^{n-1}2(2x_1+2x_2+\ldots +2 x_{n-1}+y_i))+(2x_1+2x_2+\ldots +2 x_{n-1}+y_n)$$ $$=(\sum_{i=1}^{n-1}2x_i+2y_i)+y_n~~~~~~~~~~~~~~~~~~~~~~~~~~~~~~~~~~~~~~~~~~~~~~~~~~~~~$$  $$=[x_1,\ldots,x_{n-1},[y_1,\ldots,y_n]_M]_M.~~~~~~~~~~~~~~~~~~~~~~~~~~~~~~~~~~~~~~~~~$$ Therefore (2.1) is satisfied.
One easily checks the other axioms.
\end{example}

\begin{example}
Let  $\Gamma:=\bz[t^{\pm 1}, s]/(s^2+ts-s).$ Any $\Gamma$-module $M$  endowed with the operation $[-,\ldots,-]_M$ defined by   $$[x_1,\ldots,x_n]_M=sx_1+sx_2+\ldots +s x_{n-1}+tx_n $$ is a  $n$-rack that generalizes the Alexander quandle when $s=1-t.$
Indeed $$[[x_1,\ldots,x_{n-1},y_1]_M,\ldots[x_1\ldots,x_{n-1},y_n]_M,]_M=~~~~~~~~~~~~~~~~~~~~~~~~~~~~~~~~~~~~~~~~~~$$
$$~~~~~~~~=(\sum_{i=1}^{n-1}s(sx_1+sx_2+\ldots +s x_{n-1}+ty_i))+t(sx_1+sx_2+\ldots +s x_{n-1}+ty_n)$$ $$=(s^2+st)(\sum_{i=1}^{n-1}x_i)+ts(\sum_{i=1}^{n-1} y_i)+t^2y_n~~~~~~~~~~~~~~~~~~~~~~~~~~~~~~~~~~~~$$  $$=[x_1,\ldots,x_{n-1},[y_1,\ldots,y_n]_M]_M ~~ \mbox{since}~~s^2+st=s.~~~~~~~~~~~~~~~~~~~$$ Therefore (2.1) is satisfied.
One easily checks the axiom 2. 
\end{example}
Note that for $t=1$ and $s=2,$ this coincides with example 2.3.
\begin{example}

A group G endowed with the operation $[-,\ldots,-]_G$ defined by   $$[x_1,\ldots,x_n]_G=x_1x_2\ldots x_{n-1}x_nx^{-1}_{n-1}x^{-1}_{n-2}...x_1^{-1},$$ is a pointed weak $n$-quandle (pointed by $1\in G$). Indeed,
$$[x_1,\ldots,x_{n-1},[y_1,\ldots,y_n]_G]_G=x_1x_2\ldots x_{n-1}[y_1,\ldots,y_n]_G x^{-1}_{n-1}x^{-1}_{n-2}...x_1^{-1}~~~~~~~~~~~~~~~~~$$ $$~~~~~~~~~~~~~~~~~~~~~=x_1x_2\ldots x_{n-1}y_1y_2\ldots y_{n-1}y_ny^{-1}_{n-1}y^{-1}_{n-2}\ldots y_1^{-1}x^{-1}_{n-1}x^{-1}_{n-2}...x_1^{-1}.$$
 on the other hand, 

$$[[x_1,\ldots,x_{n-1},y_1]_G,\ldots[x_1\ldots,x_{n-1},y_n]_G,]_G=~~~~~~~~~~~~~~~~~~~~~~~~~~~~~~~~~~~~~$$
$$~~~~~~~~~~~~~~~~~=[x_1x_2\ldots x_{n-1}y_1x^{-1}_{n-1}x^{-1}_{n-2}...x_1^{-1},\ldots,x_1x_2\ldots x_{n-1}y_nx^{-1}_{n-1}x^{-1}_{n-2}...x_1^{-1}]_G$$$$~~~~~~=x_1x_2\ldots x_{n-1}y_1y_2\ldots y_{n-1}y_ny^{-1}_{n-1}y^{-1}_{n-2}\ldots y_1^{-1}x^{-1}_{n-1}x^{-1}_{n-2}...x_1^{-1}$$ by cancellation. Therefore (2.1) is satisfied.
One easily checks the other axioms.

\end{example}

This determines a functor $\F: GROUP\longrightarrow~ _npRACK$ from the category of groups to the category of pointed $n$-racks. The functor $\F$ is faithful and  has a left adjoint $\F'$ defined as follows: Given a pointed $n$-rack $R,$  one constructs a Group $$G_R=<R>/I$$ where $<R>$ stands for the   free group on the elements of $R$ and $I$ is the normal subgroup generated by the set $$\{(x_1^{-1}x_2^{-1}\ldots x^{-1}_{n-1}x_n^{-1}x_{n-1}x_{n-2}\ldots x_1)([x_1,\ldots,x_n]_R)~\mbox{with}~x_i\in R, i=1,2,\ldots, n\}.$$ 
%modulo the relations  $$[x_1,\ldots,x_n]_R=x_1x_2\ldots x_{n-1}x_nx^{-1}_{n-1}x^{-1}_{n-2}...x_1^{-1}.$$ 

That $\F'$ is left adjoint to $\F$ is a consequence of the following proposition which extends to $n$-racks a well-known result on the category of racks.

\begin{proposition}
Let $G$ be  a group and let $R$ be a $n$-rack. For a morphism of $n$-racks $\alpha: R\longrightarrow \F(G),$ there is a unique morphism of groups $\alpha_*:\F'(R)\longrightarrow G$ such that the following diagram commute.

\[ 
\begin{diagram} 
  \node{}
  \node{\F'(R)} \arrow{e,t}{\alpha_*} 
  \node{G} 
  \node{}\\ 
  \node{} 
  \node{R} \arrow{e,t}{\alpha} \arrow{n,r}{}
  \node{\F(G)} \arrow{n,r}{id}
  \node{}\\
\end{diagram}
\]  
\end{proposition}
\begin{proof}
By the universal property of free groups, there is a unique morphism of groups $\beta:<R>\longrightarrow G$ such that $\alpha=\beta|_{R}$ . In particular,  for all $x_i\in R, i=1,2\ldots, n$ $$\beta((x_1^{-1}x_2^{-1}\ldots x^{-1}_{n-1}x_n^{-1}x_{n-1}x_{n-2}\ldots x_1)([x_1,\ldots,x_n]_R))=~~~~~~~~~~~~~~~~~~~~~~~~~~~~~~~~~~~~~~~~~~~~~$$$$~~~~~~~~~~~~~~~~~~~~~~~=\alpha((x_1^{-1}x_2^{-1}\ldots x^{-1}_{n-1}x_n^{-1}x_{n-1}x_{n-2}\ldots x_1)([x_1,\ldots,x_n]_R))=1.$$The result follows by the universal property of quotient groups.

\end{proof}

\begin{example}
Any rack $(R,\circ,1)$ is also a $n$-rack under the  $n$-ary operation defined by 
$$[x_1,x_2,\ldots,x_n]_R=x_1\circ(x_2\circ(\ldots(x_{n-1}\circ x_n)\ldots)).$$
This process determines a functor $\G: pRACK\longrightarrow\  _npRACK$, which has as left adjoint, the functor $\G': ~ _npRACK\longrightarrow pRACK$ defined as follows: Given a pointed $n$-rack  $(R,[-\ldots-],1),$ then $R^{\times (n-1)}$ endowed with the binary operation $$(x_1, x_2,\ldots, x_{n-1})\circ(y_1,y_2,\ldots y_{n-1})=([x_1,\ldots,x_{n-1},y_1]_R,\ldots,[x_1,\ldots,x_{n-1},y_{n-1}]_R)~~~~~~(2.2)$$ is a rack pointed at $(1,1,\ldots,1)$. Let us observe that if $R$ is a $n$-quandle, then $R^{\times (n-1)}$ is a quandle.
\end{example}

\begin{definition} 
Let $R$ be a pointed $n$-rack and let $S_R=\{f:R\longrightarrow R,~f \mbox{is a bijection}\}.$ \\Then
define $\phi: R\times R\times\ldots \times R\longrightarrow~ _nAut(R)$ by $$\phi(x_1,\ldots,x_{n-1})(y)=[x_1\ldots,x_{n-1},y]_{R}~~\mbox{ for all}~~ y\in R$$ where
 $$_nAut(R)=\{\xi\in S_R~/~ \xi([x_1,\ldots,x_n]_{R})=[\xi(x_1),\ldots,\xi(x_{n})]_{R}\}.$$
That $\phi$ is well-defined is a direct consequence of the axiom 2 of definition 2.1.

\end{definition}

\begin{proposition}

Let $(R,[-,\ldots,-]_R,1)$ is a $n$-rack, then for all $x_1,\ldots x_{n-1}\in R,$ $\phi(x_1,\ldots,x_{n-1})$ operates on $R$ by $n$-rack automorphism, i.e. $\phi(x_1,\ldots,x_{n-1})\in~ _nAut(R).$ 
\end{proposition}

\begin{proof}: $$\phi(x_1,\ldots,x_{n-1})([y_1,
\ldots,y_n]_R)=[x_1,\ldots,x_{n-1},[y_1,\ldots,y_n]_R]_R~~~~~~~~~~~~~~~~~~~~~~~~~~~$$$$~~~~~~~~~~~~
~~~~~~~~~~~~~~~~~~~~~~~~~~~~~~~~~~~~~~~~~~~~~~~~~~~=[[x_1,\ldots,x_{n-1},y_1]_R,\ldots,[x_1,\ldots,x_{n-1},y_n]_R,]_R~~\mbox{by}~ (2.1)$$$$~~~~~~~~~~~~~~~~~~~~~~~~~~~~~~~~~~=[\phi(x_1,\ldots,x_{n-1})(y_1),\ldots,\phi(x_1,\ldots,x_{n-1})(y_n)]_R$$
\end{proof}

\section{A (co)homology theory on $n$-Racks}

Recall that for a rack $(X,\circ),$ one defines \cite{CJKLS} the rack homology $H_*^R(X)$ of $X$ as the homology of the chain complex $\{C^R_k(X), \p_k\}$ where $C^R_k(X)$ is the free abelian group generated by $k$-uples $(x_1,x_2,\ldots,x_k)$ of elements of $X$ and the boundary maps $\p_k: C_k^R(X)\longrightarrow C^R_{k-1}(X)$ are defined by $$\p_k(x_1,x_2,\ldots,x_k)=\sum_{i=2}^k(-1)^i((x_1,x_2,\ldots,x_{i-1},x_{i+1},\ldots,x_k)$$$$~~~~~~~~~~~~~~~-(x_1\circ x_i,x_2\circ x_i,\ldots,x_{i-1}\circ  x_i,x_{i+1},\ldots,x_k))$$ for $k\geq 2$ and $\p_k=0$ for $k\leq 1$. If $X$ is a quandle, the subgroups $C_k^D(X)$ of $C^R_k(X)$ generated by $k$-tuples $(x_1,x_2,\ldots,x_k)$ with $x_i=x_{i+1}$ for some $i, ~1\leq i < k$ form a subcomplex $C_*^D(X)$ of $C_*^R(X)$ whose homology $H_*^D(X)$ is called the degeneration  homology of $X.$ The homology $H_*^Q(X)$ of the quotient complex $\{C_k^Q(X)=C_k^R(X)/C_k^D(X),\p_k\}$ is called the quandle homology of $X.$

Now let $\x$ be  a $n$-rack. We showed in example 2.7 that $\x^{\times(n-1)}$ endowed with the binary operation defined by (2.2) is a rack. $\x^{\times(n-1)}$ is a quandle if $\x$ is a $n$-quandle.
\begin{definition}
We define the chain complexes $_nC_*^R(\x):=C_*^R(\x^{\times(n-1)})$ if  $\x$ is a $n$-rack, $_nC_*^D(\x):=C_*^D(\x^{\times(n-1)})~\mbox{ and}~~  _nC_*^Q(\x):=C_*^Q(\x^{\times(n-1)})~~\mbox{if } \x ~\mbox{is a }~n-\mbox{quandle}.$ 
\end{definition}
\begin{definition}

Let $\x$ be a $n$-rack. The $k$th $n$-rack homology group of $\x$ with trivial coefficient is defined by $$H_k^R(\x)=H_k(_nC_*^R(\x)).$$
\end{definition}
\begin {definition}
Let $\x$ be a $n$-quandle.
\begin{enumerate} 
\item The $k$th $n$-degeneration homology group of $\x$ with trivial coefficient is defined by $$H_k^D(\x)=H_k(_nC_*^D(\x)).$$
\item The $k$th $n$-quandle homology group of $\x$ with trivial coefficient is defined by $$H_k^Q(\x)=H_k(_nC_*^Q(\x)).$$
\end{enumerate}
\end{definition}
\begin{definition}
Let $A$ be a abelian group, we define the chain complexes $$_nC_*^W(\x;A)= ~_nC_*^W(\x)\otimes A, ~~\p=\p\otimes id  ~~\mbox{with}~~ W=D,R,Q.$$

\begin{enumerate}
\item The $k$th $n$-rack homology group of $\x$ with coefficient in $A$ is defined by $$H_k^R(\x;A)=H_k(_nC_*^R(\x;A)).$$

\item The $k$th $n$-degenerate homology group of $\x$ with coefficient in $A$ is defined by $$H_k^D(\x;A)=H_k(_nC_*^D(\x;A)).$$
\item The $k$th $n$-quandle homology group of $\x$ with coefficient in $A$ is defined by $$H_k^Q(\x;A)=H_k( _nC_*^Q(\x;A)).$$
\end{enumerate}

\end{definition}

One defines the cohomology theory of $n$-racks and $n$-quandles by duality. Note that for $n=2,$ one recovers the homology and cohomology theories defined by Carter, Jelsovsky, Kamada, Landford and Saito \cite{CJKLS}.

\begin{remark}
Since $\x^{\times(n-1)}$ carries most of the  properties of $\x,$ several results established on racks are valid on $n$-racks. For instance; if $\x$ is finite, then $\x^{\times(n-1)}$ is also finite. Cohomology of finite racks were studied by  Etingof and  Gra\~na in \cite{EG}
\end{remark}

\section{From Lie $n$-racks to Leibniz $n$-algebras}

In this section we define the notion of Lie $n$-racks and provide a connection with Leibniz $n$-algebras.  Throughout the section, $T_1$ denotes the tangent functor.

\begin{definition}

A Lie $n$-rack   $(R,[-,\ldots,-]_{R},1)$ is a smooth manifold  $R$ with the structure of a pointed $n$-rack such that the  $n$-ary operation $[-,\ldots,-]_{R}: R\times R\times\ldots \times R\longrightarrow R$  is a smooth mapping. For $n=2,$ one recovers Lie racks \cite{A}.

\end{definition}

\begin{example}

Let   $H$ be a Lie group. Then H endowed with the operation    $$[x_1,\ldots,x_n]_G=x_1x_2\ldots x_{n-1}x_nx^{-1}_{n-1}x^{-1}_{n-2}...x_1^{-1},$$ is a Lie $n$-rack.

\end{example}

\begin{example}

Let $(H, \{-\ldots-\})$ be a  group endowed with an antisymmetric $n$-ary operation, and $V$ an $H$-module. Define the $n$-ary operation $ [-,\ldots,-]_{R}$ on $R:=V\times H$ by $$[(u_1, A_1), (u_2, A_2)\ldots (u_n, A_n)]_{R}:=(\{A_1,\ldots, A_n\}u_n, A_1A_2\ldots A_{n-1}A_nA^{-1}_{n-1}A_{n-2}^{-1}\ldots A_1^{-1}).$$ Then $ (R, [-\ldots-]_R, (0,1)) $ is a Lie $n$-rack.

\end{example}

\begin{theorem}

Let $R$ be a Lie $n$-rack and $g:=T_1R.$  For all  $ x_1, x_2, \ldots, x_{n-1}\in R,$ the tangent mapping $\Phi (x_1,x_2,\ldots,x_{n-1})=T_1(\phi(x_1,x_2,\ldots,x_{n-1}))$ is an automorphism of $(\g, [-,\ldots,-]_{\g}).$
\end{theorem}

\begin{proof}
Since   $\phi(x_1,x_2,\ldots,x_{n-1})(1)=[x_1, x_2,\ldots,x_{n-1},1]_R=1,$ we apply the tangent functor $T_1$ to $\phi(x_1,x_2,\ldots,x_{n-1}): R\longrightarrow R$ and obtain 
$\Phi(x_1,x_2,\ldots,x_{n-1}): T_1R\longrightarrow T_1R$ which is in $GL(T_1 R)$ as $\phi(x_1,x_2,\ldots,x_{n-1})\in~ _nAut(R)$  by proposition 2.9. Now by the left distributive property of  $n$-racks, we have
$$\phi(x_1,x_2,\ldots,x_{n-1})(\phi(y_1,y_2,\ldots,y_{n-1})(y_n))=~~~~~~~~~~~~~~~~~~~~~~~~~~~~~~~~~~~$$$$=\phi(\phi(x_1,\ldots,x_{n-1})(y_1),\phi(x_1,\ldots,x_{n-1})(y_2),\ldots,\phi(x_1,\ldots,x_{n-1})(y_{n-1}))(\phi(x_1,\ldots,x_{n-1})(y_n))$$ which successively differentiated at $1\in R$ with respect  to $y_n,$ then $y_{n-1},$  until $y_1$ yields to 
$$\Phi(x_1,x_2,\ldots,x_{n-1})([Y_1,Y_2,\ldots, Y_n]_{\g})=~~~~~~~~~~~~~~~~~~~~~~~~~~~~~~~~~~~~~~~$$$$=[\Phi(x_1,x_2,\ldots,x_{n-1})(Y_1),\Phi(x_1,x_2,\ldots,x_{n-1})(Y_2),\ldots, \Phi(x_1,x_2,\ldots,x_{n-1})(Y_n)]_{\g}~~~(4.1)$$  for all  $Y_1, Y_2, \ldots, Y_n \in \g.$
\end{proof}

\begin{theorem}
 
Let $R$ be a Lie $n$-rack and  let $x_1, \ldots, x_{n-1}\in R$ corresponding respectively to $X_1,\ldots,X_{n-1}\in \g:=T_1R.$ Then, the  adjoint derivation  $ad\{X_1, \ldots, X_{n-1}\}: \g\longrightarrow gl(\g)$ defined by $$ad\{X_1, X_2, \ldots, X_{n-1}\}(Y)=[X_1,X_2,\ldots, X_{n-1}, Y]_{\g}$$ is exactly $T_1(\Phi)$

\end{theorem}

\begin{proof}

From the proof of theorem 4.4, $\Phi(x_1,x_2,\ldots,x_{n-1})\in GL(\g).$ Also, the mapping $\Phi: R\times R\times\ldots\times R\longrightarrow GL(\g)$ satisfies $\Phi(1,1,\ldots,1)=I,$ where $I\in GL(\g)$ is the identity. Differentiating $\Phi$ at (1,1,\ldots,1) yields a mapping  $T_1(\Phi): T_1( R\times R\times\ldots\times R)\longrightarrow gl(\g),$ where $gl(\g)$ is the Lie algebra  associated to the Lie group $GL(\g).$ Also Differentiating the identity (4.1) at  $(1,1,\ldots,1)$ with respect to  $(x_1,x_2,\ldots,x_{n-1})$  yields  $$[X_1,\ldots, X_{n-1}, [Y_1,Y_2,\ldots, Y_n]_{\g}]_{\g}=\sum_{i=1}^n[Y_1,\ldots,Y_{i-1},[X_1,\ldots,X_{n-1},Y_i]_{\g}, Y_{i+1},\ldots, Y_n]_{\g}.$$ Hence the mapping 
$T_1(\Phi)$ is the adjoint derivation.

\end{proof}

\begin{corollary}

Let $R$ be a Lie $n$-rack and $\g:=T_1R.$
Then there exist an $n$-linear mapping $[-,\ldots,-]_{\g}: \g\times \g \times\ldots\times\g\longrightarrow\g $ such that $(\g, [-,\ldots,-]_{\g})$ is a Leibniz $n$-algebra.

\end{corollary}

\begin{proof}
From  the proofs of theorems 4.4 and 4.5, it is clear that the $n$-ary operation  $[-,\ldots,-]_{\g}$ is a derivation for itself. 
\end{proof}


\begin{thebibliography}{aaaa}

\bibitem{A} Andruskiewitsch, N. and M. Gra\~na, \textit{From racks to pointed Hopf algebras,} Adv. Math.  \textbf{178}   (2003) 177-243.   


%\bibitem{AI} de Azc�arraga, J. A., Izquierdo, J. M., P�erez Bueno, J. C.,``On the higher-order generalizations of Poisson structures,'' \textit{J. Phys. A: Math.} Gen. \textbf{30}   (1997) L607-L616.   
\bibitem{B} Bagger, J., N. Lambert,  \textit{Gauge Symmetry and Supersymmetry of Multiple
M2-Branes},Phys. Rev. \textbf{D77}(6), 2008.
\bibitem{CW} Conway J., G. Wraith, Unpublished correspondence (1958).
\bibitem {CJKLS}  Carter J. S., D. Jelsovsky, S. Kamada, L. Langford, M. Saito,\textit{ Quandle cohomology
and state-sum invariants of knotted curves and surfaces,} Trans.
Amer. Math. Soc. \textbf{355}, 10 (2003), 3947-3989.


%\bibitem{CJ}  Casas, J. M., ``Homology with coefficients of Leibniz $n$-algebras'', \textit{C. R. Acad. Sci. Paris}.  \textbf{Ser 1} 347 (2009) 595-598.
\bibitem{CLP} Casas, J. M.,  J. L.  Loday, T.  Pirashvili, \textit{Leibniz $n$-algebras}, Forum Math. \textbf{14},  (2002), 189-207.
\bibitem{C}  Covez, S., \textit{The local integration of Leibniz algebras }, 2010, arXiv:1011.4112v1
\bibitem{EG} Etingof P.,  M. Gra\~na, \textit{ On rack cohomology,} J. Pure Appl. Algebra \textbf{177} (2003), no. 1, 49-59.
\bibitem{FRS}  Fenn R., C. Rourke, B. Sanderson, \textit{James bundles and applications,} Proc. London Math. Soc. \textbf{3},
89 (2004), no. 1, 217-240.
\bibitem{F} Filippov, V. T., \textit{$n$-Lie algebras,} Sibirsh. Mat. Zh. \textbf{26}, 6 (1985):126-140.


\bibitem{FU} Furuuchi, K., D. Tomino, \textit{Supersymmetric reduced models with a symmetry based on Filippov algebra}, 2009. arXiv:0902.2041v3.
\bibitem{GM} Grabowski J., G.  Marmo, \textit{On Filippov algebroids and multiplicative Nambu - Poisson structures,} Diff. Geom. Appl.  \textbf{12}   (2000) 35-50.   
\bibitem{DJ} Joyce D., \textit{A classifying invariant of knots: the knot quandle,} Journal of Pure and Applied Algebra  \textbf{23}   (1982) 37-65.   
\bibitem {K} Kinyon, K., \textit{Leibniz algebra, Lie racks and digroups,} J.Lie Theory, \textbf{17} (2007) 99-11.
\bibitem{L} Loday, J.-L.,\textit{Cyclic Homology}, Springer-Verlag, Berlin, Heidelberg, New York, 1992.
%\bibitem{LP} Loday, J.-L., Pirashvili, T., ``Universal enveloping algebras of Leibniz algebras and (co)-homology,'' \textit{Math. Annalen,} \textbf{296}, 1 (1993), 139-158.

\bibitem{LD} Loday, J.-L.,  \textit{Une Version Non Commutative des Algebres de Lie: Les Algebres de Leibniz,} L'Enseignement Math. \textbf{39}, (1993), 269-293.
\bibitem{N} Nambu, Y., \textit{Generalized Hamiltonian mechanics}, Phys. Rev. \textbf{D7} (1973), 2405-2412
%\bibitem{TP} Pirasvili, T., ``On Leibniz Homology,'' \textit{Annales de l'institut Fourrier}, Grenoble \textbf{44}, 2 (1994), 401-411.
\bibitem {TA} Takasaki, M.,\textit{ Abstraction Of Symmetric Transformation,}  Tohoku Math. J. \textbf{49} (1942/3) 145-207.
\bibitem{T} Takhtajan L. A., \textit{On foundations of generalized Nambu mechanics,} Comm. Math. Phys.  \textbf{160}   (1994) 295-315.   

\bibitem{V} Vaisman I., \textit{Nambu- Lie groups,} J. Lie Theory  \textbf{10}   (2000) 181-194.   




\end{thebibliography}
\end{document}